%% file: Minimal_dilatation_in_Penner_s_construction.tex
\documentclass[11pt]{amsart}

\usepackage{amsmath}
\usepackage{amssymb}
\usepackage{graphicx}
\usepackage{color}
\usepackage{url}

\newtheorem{thm}{Theorem}

\newtheorem{que}[thm]{Question}
\newtheorem{prop}[thm]{Proposition}
\newtheorem{lem}[thm]{Lemma}
\newtheorem{remark}[thm]{Remark}

\newtheorem{ex}[thm]{Example}

\newcommand{\R}{\mathbf{R}}

\newcommand{\Ss}{\mathbf{S}} 

\begin{document}
\title{Minimal dilatation in Penner's construction}
\author{Livio Liechti}
\thanks{The author is supported by the Swiss National Science Foundation (project no.\ 159208)}
\address{Mathematisches Institut der Universit\"at Bern, Silderstrasse~5, CH-3012 Bern, Switzerland}
\email{livio.liechti@math.unibe.ch}

\begin{abstract} 
For all orientable closed surfaces, we determine the minimal dilatation among mapping classes arising from Penner's construction.
We also discuss generalisations to surfaces with punctures.
\end{abstract}
\maketitle

\section{Introduction}
Thurston's famous classification states that an irreducible mapping class is either periodic or pseudo-Anosov~\cite{Th}.
A \emph{mapping class} $\phi$ is a diffeomorphism of a surface of finite type, up to isotopy relative to the boundary. 
It is \emph{periodic} if it has a periodic representative and \emph{pseudo-Anosov} if it has a pseudo Anosov representative, 
i.e.\ there exist two invariant transverse measured foliations such that this representative stretches one of them by some real number $\lambda>1$ and the other by $\lambda^{-1}$. 
The number $\lambda$ is the \emph{dilatation} of a pseudo-Anosov mapping class $\phi$.

In this article, we consider dilatations arising from Penner's construction of pseudo-Anosov mapping classes~\cite{Pe}.
By a result of Leininger, the dilatation of any such mapping class is bounded from below by $\sqrt{5}$, see the appendix of~\cite{Le}. 
However, Leininger states that this bound is not sharp.
For every orientable closed surface, %$\Sigma_g$ of genus~$g$ 
we give the optimal lower bound 
and determine a pseudo-Anosov mapping class %$\phi_g$ 
arising via Penner's construction realising it.

\begin{thm} 
\label{minimaldilatation}
On an orientable closed surface $\Sigma_g$ of genus $g\ge1$, the minimal dilatation among mapping classes arising from Penner's construction is 
$$\lambda_g = 2 - \text{\emph{cos}}\left(\frac{2g-1}{2g+1}\pi\right) + \sqrt{3 -4\text{\emph{cos}}\left(\frac{2g-1}{2g+1}\pi\right) +\text{\emph{cos}}^2\left(\frac{2g-1}{2g+1}\pi\right)}.$$
Furthermore, the dilatation $\lambda_g$ is realised by the Coxeter mapping class associated to the Coxeter graph $(A_{2g},\pm)$ with alternating signs.
\end{thm}

Recent results deal with Galois conjugates of dilatations arising from Penner's construction. For example, they lie dense in the complex plane by a theorem of Strenner~\cite{Str}.
On the other hand, Shin and Strenner showed that they cannot lie on the unit circle and used their result to disprove Penner's conjecture that every pseudo-Anosov mapping class has a power arising via his construction~\cite{ShSt}.

\begin{remark} \emph{The sequence $\lambda_g$ of minimal dilatations among mapping classes arising via Penner's construction for $g\ge1$ is monotonically increasing in $g$. 
This follows directly from the formula given in Theorem~\ref{minimaldilatation}, since $-\text{cos}(\frac{2g-1}{2g+1}\pi)$ is monotonically increasing in $g\ge1$.
Thus, the minimum among all dilatations arising from Penner's construction is $$\lambda_1=\frac{3+\sqrt{5}}{2},$$ 
the square of the golden ratio, and is geometrically realised by the monodromy of the figure eight knot. 
Furthermore, since $-\text{cos}(\frac{2g-1}{2g+1}\pi)$ converges to $1$ as $g\to +\infty$, the sequence $\lambda_g$ 
converges to the limit $$\lim_{g\to\infty}\lambda_g = 3 + 2\sqrt{2},$$ 
the square of the silver ratio.
}\end{remark}

\begin{remark} \emph{One way to describe the minimising Coxeter mapping class $\phi_g$ with alternating-sign Coxeter graph $(A_{2g},\pm)$ is the following: 
Let $\Sigma_g$ be a closed surface of genus $g$. Let $\alpha$ and $\beta$ be two multicurves in $\Sigma_g$, 
with $g$ components each, that intersect with the pattern of the Dynkin tree $A_{2g}$.
For $g=3$, this is depicted in Figure~\ref{g3}. 
\begin{figure}[h]
\def\svgwidth{250pt}
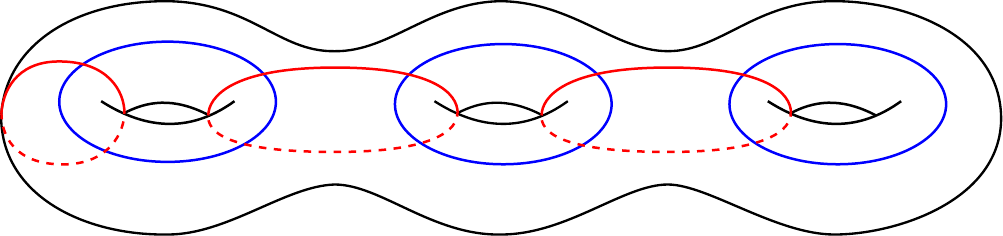
\caption{}
\label{g3}
\end{figure}
Then $\phi_{g}$ is the mapping class given by a negative Dehn twist along all components of $\beta$, 
followed by a positive Dehn twist along all components of $\alpha$. 
Manifestly, the mapping classes $\phi_g$ also arise via Thurston's construction~\cite{Th}.
The connection to Coxeter mapping classes will be made more precise in Section~\ref{Coxeter_section}.
}\end{remark}

\subsection{Penner's construction}
\label{Pennersection}
Let $\Sigma$ be an oriented surface of finite type and let $\gamma\subset\Sigma$ be a simple closed curve.
A \emph{Dehn twist along} $\gamma$ is a mapping class with 
a representative supported in an annular neighbourhood of $\gamma$, sending an arc crossing 
$\gamma$ to an arc crossing $\gamma$ but also winding around $\gamma$ once.
Dehn twists come in two flavours: \emph{positive} if the winding is in the counterclockwise sense and \emph{negative} if the winding is in the clockwise sense.
This does not depend on the orientation of $\gamma$ but only on the orientation of~$\Sigma$. 
We write $T_\gamma^+$ for the positive and $T_\gamma^-$ for the negative Dehn twist along $\gamma$, respectively.
A \emph{multicurve} $\gamma$ on $\Sigma$ is a disjoint union of simple closed curves $\gamma_i \subset \Sigma$.
We now describe Penner's construction. Let $\alpha=\alpha_1~\dot{\cup}~\cdots~\dot\cup~\alpha_n$ and 
$\beta=\beta_1~\dot\cup~\cdots~\dot\cup~\beta_m$ be two multicurves (without parallel components) which intersect minimally and whose union \emph{fills} $\Sigma$, 
i.e.\ the complement consists of discs and once-punctured discs.
Furthermore, let $\phi$ be a product of positive Dehn twists $T_{\alpha_i}^+$ and negative Dehn twists $T_{\beta_j}^-$ 
such that every component of the multicurves $\alpha$ and $\beta$ gets twisted along at least once. By a theorem of Penner, such a mapping class $\phi$ is pseudo-Anosov
and the dilatation of $\phi$ can be obtained as follows~\cite{Pe}. 
Let 
\begin{align*}
M_{\alpha_i} = I + R_{\alpha_i},\\
M_{\beta_j} = I + R_{\beta_j},
\end{align*}
where $I$ is the identity matrix of size $(n+m)\times(n+m)$. 
Furthermore, the matrices $R_{\alpha_i}$ and $R_{\beta_j}$ are obtained from the
geometric intersection matrix 
of the curves $\{\alpha_1,\dots,\alpha_n,\beta_1,\dots,\beta_m\}$,
$$\begin{pmatrix}
  0 & X\\
 X^{\top} & 0
 \end{pmatrix}\ge0,
$$ 
by setting all entries to zero which are not in the row corresponding to $\alpha_i$ or $\beta_j$, respectively.
Then the dilatation $\lambda$ of $\phi$ is the Perron-Frobenius eigenvalue of the matrix product $M_\phi\ge0$ corresponding to the product $\phi$ of Dehn twists~\cite{Pe}.

\subsection{Outline} As a first step in the proof of Theorem~\ref{minimaldilatation}, we recall the important notions concerning Coxeter mapping classes in Section~\ref{Coxeter_section}. 
This will provide us with examples of mapping classes arising via Penner's construction. In particular, we will calculate the dilatation of the Coxeter mapping class 
corresponding to the alternating-sign Coxeter graph $(A_{2g},\pm)$ in Proposition~\ref{A2gdilatation}. 
Having calculated these dilatations allows us to neglect all mapping classes for which we can deduce larger dilatation.  
Section~\ref{bounds_section} is devoted to this task: in Proposition~\ref{doubleintersection}, we show that we can disregard the case where two components of the multicurves 
used in Penner's construction intersect more than once, essentially reducing the problem to a question about alternating-sign Coxeter mapping classes. 
Using monotonicity of the spectral radius under Coxeter graph inclusion in Proposition~\ref{affinelemma}, we are able to exclude almost all Coxeter graphs that 
do not correspond to a finite Dynkin diagram. 
For pairs of multicurves that intersect with the pattern of a graph that we did not rule out, we study in Section~\ref{filling_section} which surfaces their union can fill.
This finally allows us to finish the proof of Theorem~\ref{minimaldilatation} by calculating the dilatations of very few small genus examples.
We round off in Section~\ref{punctures}, where we hint at generalisations to surfaces with punctures, discuss difficulties and ask about asymptotic behaviour.
\medskip

\noindent
{\bf Acknowledgements.} I would like to thank Sebastian Baader, Pierre Dehornoy and Eriko Hironaka for insightful discussions and suggestions. 
I would also like to thank Bal\'azs Strenner and the referee for helpful comments.

\section{Coxeter mapping classes}
\label{Coxeter_section}

Let $\Gamma$ be a finite connected graph without loops or double edges and let $\mathfrak{s}$ be an assignment of a sign $+$ or $-$ to every vertex $v_i$ of $\Gamma$. 
Such a pair $(\Gamma, \mathfrak{s})$ is called a \emph{mixed-sign Coxeter graph}, cf.~\cite{Hi}. 
To such a pair, we will associate certain products of reflections. 
Let $\R^{V_\Gamma}$ be the real vector space abstractly generated by the vertices $v_i$ of $\Gamma$. 
We equip $\R^{V_\Gamma}$ with the symmetric bilinear form $B$, given by $B(v_i,v_i)=-2\mathfrak{s}(v_i)$ and $B(v_i,v_j) = a_{ij}$, where 
$a_{ij}$ is the $ij$-th entry of the adjacency matrix $A(\Gamma)$ of $\Gamma$.
To every vertex $v_i$, we associate a reflection $s_i$ about the hyperplane in $\R^{V_\Gamma}$ perpendicular to $v_i$,
$$s_i(v_j) = v_j - 2\frac{B(v_i,v_j)}{B(v_i,v_i)}v_i.$$
A product of the $s_i$ containing every $s_i$ exactly once is called a \emph{Coxeter transformation} associated to the pair $(\Gamma,\mathfrak{s})$. 
For arbitrary graphs $\Gamma$, this product is highly non-unique. If $\Gamma$ is a tree, however, the Coxeter transformation is uniquely determined up to conjugation~\cite{Steinberg}.
This allows us to talk about ``the" Coxeter transformation associated to a pair $(\Gamma,\mathfrak{s})$ in case $\Gamma$ is a tree.
For our purposes, the most important examples of trees are the simply laced finite Dynkin diagrams, depicted in Figure~\ref{ADE}.

\begin{figure}[h]
\def\svgwidth{180pt}
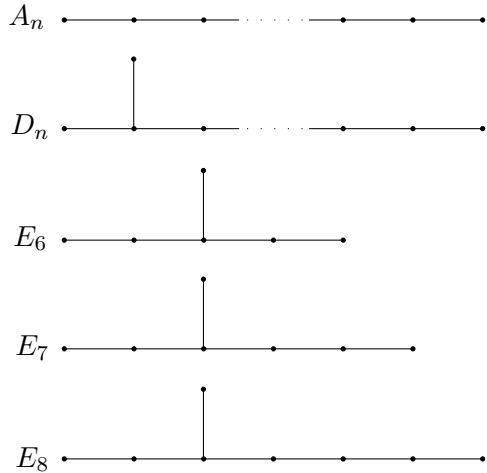
\caption{The Dynkin diagrams $A_n$, $D_n$, $E_6$, $E_7$ and $E_8$.}
\label{ADE}
\end{figure}

In the following, two kinds of sign assignments $\mathfrak{s}$ will be of interest to us: 
all signs $\mathfrak{s}$ positive, written $(\Gamma,+)$, which is the case we call \emph{classical}, 
and signs $\mathfrak{s}$ that give a bipartition of the (necessarily bipartite) graph $\Gamma$, written $(\Gamma,\pm)$, which is the case we call \emph{alternating-sign}.

In the case of classical Coxeter trees $(\Gamma,+)$, A'Campo realised the Coxeter transformation, up to a sign, as the homological action of a mapping class 
given by the product of two positive Dehn twists along multicurves that intersect each other with the pattern of $\Gamma$, see~\cite{AC2}. 
Similarly, in the case of alternating-sign Coxeter trees $(\Gamma,\pm)$, Hironaka and the author realised the Coxeter transformation, up to a sign, as the homological action of a mapping class 
given by the product of two Dehn twists of opposite signs along multicurves that intersect each other with the pattern of $\Gamma$, see~\cite{HiLi}.
We call these mapping classes \emph{Coxeter mapping classes}.
Two observations are of special interest for us. Firstly, the Coxeter mapping classes $\phi$ built by Hironaka and the author also arise via Penner's construction. 
Indeed, the union of the constructed multicurves fills the surface (which has boundary, along which we glue in discs to land in the setting we are considering) 
and components intersect at most once, hence their intersection is minimal.
Secondly, the given matrix describing the homological action of $\phi$ equals the corresponding matrix product $M_\phi$ of Penner's construction described in Section~\ref{Pennersection}, see~\cite{HiLi}.
We deduce that for such a mapping class $\phi$ realising the Coxeter transformation associated to $(\Gamma,\pm)$, the dilatation equals 
the spectral radius of the Coxeter transformation associated to $(\Gamma,\pm)$.
This will greatly simplify our calculations.

\begin{ex}\emph{
Let $(\Gamma,\pm)$ be the $4$-cycle graph with alternating signs, as depicted in Figure~\ref{4-cycle}. 
\begin{figure}[h]
\def\svgwidth{60pt}
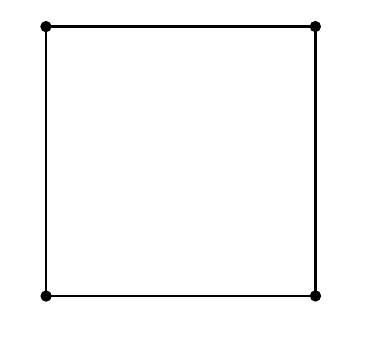
\caption{}
\label{4-cycle}
\end{figure}
There are, up to conjugation, two Coxeter transformations: 
the one corresponding to the bipartite order, $s_1s_3s_2s_4$, and the one corresponding to the cyclic order, $s_1s_2s_3s_4$. 
Two quick calculations confirm that their spectral radii are equal to and greater than $3+2\sqrt{2}$, respectively.
}\end{ex}

\begin{prop}
\label{A2gdilatation}
The Coxeter mapping class $\phi_g$ realising the Coxeter transformation of the alternating-sign Coxeter tree $(A_{2g},\pm)$ 
has dilatation 
$$2 - \text{\emph{cos}}\left(\frac{2g-1}{2g+1}\pi\right) + \sqrt{3 -4\text{\emph{cos}}\left(\frac{2g-1}{2g+1}\pi\right) +\text{\emph{cos}}^2\left(\frac{2g-1}{2g+1}\pi\right)}<3+2\sqrt{2}.$$
\end{prop}

\begin{proof}
Let $\Gamma$ be a finite tree. By a theorem of A'Campo~\cite{AC}, the eigenvalues $\mu_i$ of the classical Coxeter transformation corresponding to $(\Gamma,+)$ are related to the 
eigenvalues $\alpha_i$ of the adjacency matrix $A(\Gamma)$ of $\Gamma$ by the equation $$\alpha_i^2 = 2 + \mu_i + \mu_i^{-1}.$$
On the other hand, Hironaka and the author showed that the eigenvalues $\lambda_i$ of the Coxeter transformation 
corresponding to the Coxeter tree $(\Gamma,\pm)$ with alternating signs are related to the eigenvalues $\alpha_i$ of $A(\Gamma)$ by the equation 
$$\alpha_i^2 = -2 -\lambda_i - \lambda_i^{-1},$$ see~\cite{HiLi}.
Combined, we get that the $\mu_i$ and the $\lambda_i$ are related by 
\begin{align}
\label{coxeterrelation}
\lambda_i + \lambda_i^{-1} = -4-\mu_i-\mu_i^{-1}.
\end{align}
The transition from the Coxeter transformation to the homological action of the associated Coxeter mapping class requires the addition of a minus sign, see~\cite{AC2, HiLi}.
We slightly abuse notation and again write $\mu_i$ and $\lambda_i$ for the eigenvalues of the homological actions of the associated Coxeter mapping classes. 
Including the necessary minus signs, the relation~(\ref{coxeterrelation}) translates to
\begin{align}
\label{relation_eq}
\lambda_i + \lambda_i^{-1} = 4 - \mu_i - \mu_i^{-1},
\end{align}
which is a relation we will make good use of.
The Coxeter mapping class corresponding to the classical Coxeter tree 
$(A_{2g},+)$ is the monodromy of the torus knot $T(2,2g+1)$, see~\cite{AC2}.
On the other hand, the characteristic polynomial of the homological action of the monodromy equals the Alexander polynomial $\Delta(t)$, see e.g.~\cite{Rolfsen}, 
which is well-known for torus knots $T(2,2g+1)$:  
$$\Delta(t) = \frac{(t^{4g+2}-1)(t-1)}{(t^{2g+1}-1)(t^2-1)}.$$
All roots of $\Delta(t)$ lie on the unit circle. 
It follows that the right side of~(\ref{relation_eq}) is real. Furthermore, it is maximised for the root of $\Delta(t)$ with smallest real part, which is 
$\mu = \xi_{4g+2}^{2g-1},$
where $\xi_{4g+2} = \exp{(\frac{2\pi i}{4g+2})}\in\Ss^1.$ 
The real part of $\mu$ is 
$$\text{Re}(\mu) = \text{cos}\left(\frac{2g-1}{2g+1}\pi\right) > -1.$$
Now let $\lambda$ be the largest root of the homological action of the Coxeter mapping class associated to the Coxeter tree $(A_{2g},\pm)$ with alternating signs. 
By~(\ref{relation_eq}), we obtain the quadratic equation $$\lambda + \lambda^{-1} = 4 - 2\text{Re}(\mu),$$ 
whose larger solution $\lambda$ is given by 
$$\lambda = 2 - \text{Re}(\mu) + \sqrt{3 -4\text{Re}(\mu) +\text{Re}(\mu)^2}<3+2\sqrt{2}.$$
Inserting the value of $\text{Re}(\mu)$ given above yields the claimed result.
\end{proof}

\section{Two dilatation bounds for Penner's construction}
\label{bounds_section}

As a product of non-negative matrices of the form $(I + R_{\alpha_i})$ and $(I + R_{\beta_j})$, the matrix product $M_\phi$ of Penner's construction described in Section~\ref{Pennersection} 
is non-negative as well.
This allows us to use Perron-Frobenius theory. 
The one standard fact we repeatedly use is the following. If a non-negative matrix $M$ is entrywise greater than or equal to another non-negative matrix $N$, written $M\ge N$, 
then the spectral radius of $M$ is greater than or equal to the spectral radius of $N$.
We directly observe that the dilatation among mapping classes arising from 
Penner's construction is minimised by products of Dehn twists such that every component of the multicurves $\alpha$ and $\beta$ gets twisted along exactly once.
Propositions~\ref{doubleintersection} and~\ref{affinelemma} give lower bounds for the dilatation of mapping classes arising from Penner's construction using certain pairs of multicurves $\alpha$ and $\beta$.   

\begin{prop}
\label{doubleintersection}
In Penner's construction,
if two components $\alpha_i$ and $\beta_j$ of the multicurves $\alpha$ and $\beta$ intersect at least twice, 
then any resulting pseudo-Anosov mapping class has dilatation greater than or equal to $3+2\sqrt{2}$.
\end{prop}

\begin{proof}
Let $\phi$ be a pseudo-Anosov mapping class arising from Penner's construction using the multicurves $\alpha$ and $\beta$. 
Denote by $M_\phi$ the matrix product associated to $\phi$ in Penner's construction. 
Furthermore, we suppose that two components $\alpha_i$ and $\beta_j$ of $\alpha$ and $\beta$ intersect $x\ge 2$ times. 
Without loss of generality, we assume that 
$T_{\alpha_i}^+$ appears in the product $\phi$ before $T_{\beta_j}^-$.
From the definition, we directly obtain
$$M_\phi \ge (I + R_{\alpha_i})(I + R_{\beta_j}).$$
Up to a change of base (permuting the basis elements such that the first two basis elements correspond to $\alpha_i$ and $\beta_j$), we have
\begin{align*}(I + R_{\alpha_i})(I + R_{\beta_j}) &=
\begin{pmatrix}
  1 & x & \ast\\
 0 & 1 & 0\\
 0 & 0 & I_{n+m-2}
 \end{pmatrix}
 \begin{pmatrix}
  1 & 0 & 0\\
 x & 1 & \ast\\
 0 & 0 & I_{n+m-2}
 \end{pmatrix}\\
 &\ge
  \begin{pmatrix}
  1+x^2 & x & 0\\
 x & 1 & 0\\
 0 & 0 & I_{n+m-2}
 \end{pmatrix}=M(x).
\end{align*}
For $x=2$, the largest eigenvalue of $M(x)$ is exactly $3+2\sqrt{2}$. 
The statement now follows from monotonicity of the spectral radius of non-negative matrices under ``$\ge$".
\end{proof}

\begin{prop}
\label{affinelemma}
In Penner's construction, if the multicurves $\alpha$ and $\beta$ intersect with the pattern of a graph  
that contains an affine Dynkin diagram $\widetilde D_n$, $\widetilde E_6$, $\widetilde E_7$ or $\widetilde E_8$ as a subgraph,
then any resulting pseudo-Anosov mapping class $\phi$ has dilatation greater than or equal to $3+2\sqrt{2}$. 
\end{prop}
\begin{figure}[h]
\def\svgwidth{175pt}
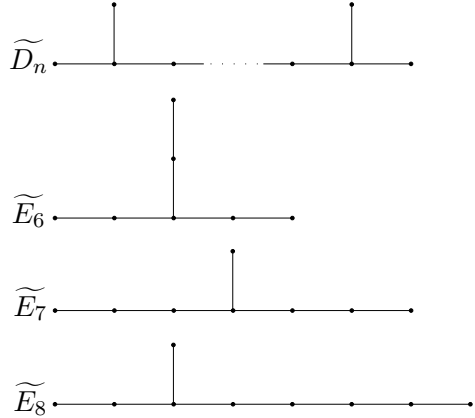
\caption{The affine Dynkin diagrams $\widetilde{D_n}$, $\widetilde{E_6}$, $\widetilde{E_7}$ and $\widetilde{E_8}$.}
\label{DEtilde}
\end{figure}
\begin{proof}
Let $\phi$ be a pseudo-Anosov mapping class arising from Penner's construction using multicurves $\alpha$ and $\beta$ that intersect with the pattern of a graph that 
contains an affine Dynkin diagram $\Gamma=\widetilde D_n, \widetilde E_6, \widetilde E_7$ or $\widetilde E_8$ as a subgraph.
As observed at the beginning of this section, we can assume 
every component of the multicurves $\alpha$ and $\beta$ to be twisted along exactly once.
Let $M_\phi$ be the corresponding matrix product described in Penner's construction. 
Furthermore, let $M_\Gamma$ be the subproduct associated to the curve components corresponding to the vertices of $\Gamma$. 
We have $M_\phi \ge M_\Gamma$. 
The spectral radius of $M_\Gamma$ is in turn an upper bound for the dilatation $\lambda$ of the alternating-sign Coxeter mapping class associated to $(\Gamma,\pm)$. 
We calculate this dilatation $\lambda$ knowing that it equals the spectral radius of the homological action and use~(\ref{relation_eq}) from the proof of Proposition~\ref{A2gdilatation}.
Since $(\Gamma,+)$ is an affine Coxeter graph, 
the homological action of the classical Coxeter mapping class associated to $(\Gamma,+)$ has all eigenvalues on the unit circle, with one eigenvalue $-1$, see~\cite{AC}. 
We obtain $\lambda+\lambda^{-1} = 6$, which yields $\lambda = 3+2\sqrt{2}$.
\end{proof}

\section{Filling pairs of multicurves}
\label{filling_section}

Let $\Sigma_g$ be an orientable closed surface of genus $g$. 
A pair of multicurves $\alpha$ and $\beta$ whose union fills $\Sigma_g$ induces a cell decomposition of $\Sigma_g$:
the $0$-cells are the intersection points of $\alpha$ and $\beta$, the $1$-cells are the connected components of $\alpha\cup\beta$ without the intersection points 
and the $2$-cells are the connected components of the complement of $\alpha\cup\beta$. 
In particular, the contribution of the $0$-cells and $1$-cells to the Euler characteristic can be directly deduced from the intersection graph of $\alpha$ and $\beta$. 
In order to know the number of $2$-cells, additional information on the framing of the curves might be necessary. 
For a pair of multicurves that intersect with the pattern of a tree, however, the number of $2$-cells does not depend on the framing and can thus be calculated directly from the tree.
For the Dynkin diagram $A_n$, the number of $2$-cells of the induced cell decomposition is two if $n$ is odd and one if $n$ is even. 
We directly deduce which closed surfaces can be filled by a pair of multicurves that intersect with the pattern of the Dynkin diagram~$A_n$.

\begin{lem}
\label{Afilling}
The union of two multicurves $\alpha$ and $\beta$ that intersect with pattern $A_{2g}$ or $A_{2g+1}$ can only fill a closed surface of genus $g$.
\end{lem}

We proceed similarly for the Dynkin diagram $D_n$: the number of $2$-cells of the induced cell decomposition is two if $n$ is odd and three if $n$ is even. 
This yields the following result.

\begin{lem}
\label{Dfilling}
The union of two multicurves $\alpha$ and $\beta$ that intersect with pattern $D_{2g+1}$ or $D_{2g+2}$ can only fill a closed surface of genus $g$.
\end{lem}

If the intersection graph of two multicurves $\alpha$ and $\beta$ contains a cycle, 
the number of $2$-cells may well depend on the framing of the curves. 
Since the Euler characteristic of an orientable closed surface is even, 
we can still deduce the parity of the number of $2$-cells directly from the graph.
For example, we obtain that the number of $2$-cells of a cell decomposition 
induced by a pair of multicurves $\alpha$ and $\beta$ that intersect like a $2n$-cycle is even and hence at least two.
Again, we directly deduce information about the genus of a surface filled that way.
\begin{lem}
\label{cyclefilling}
The union of two multicurves $\alpha$ and $\beta$ that intersect with the pattern of a $2g$-cycle can only fill closed surfaces of genus at most $g$.
\end{lem}

\subsection{Proof of Theorem~\ref{minimaldilatation}}
Let $\Sigma_g$ be an orientable closed surface of genus~$g$. As we have seen in Proposition~\ref{A2gdilatation}, there exists a mapping class $\phi_g$ on $\Sigma_g$
that arises via Penner's construction and has dilatation $\lambda_g < 3+2\sqrt{2}$. Hence, if we want to find the minimal dilatation of mapping classes on $\Sigma_g$ arising from Penner's construction, 
we can discard all pairs of multicurves $\alpha$ and $\beta$ that always yield dilatations greater than or equal to $3+2\sqrt{2}$. 
By Proposition~\ref{doubleintersection}, this excludes pairs of multicurves $\alpha$ and $\beta$ with components $\alpha_i$ and $\beta_j$ that intersect more than once. 
Furthermore, by Proposition~\ref{affinelemma}, this rules out pairs of multicurves $\alpha$ and $\beta$ that intersect with the pattern of a graph that contains
an affine Dynkin diagram $\widetilde D_n, \widetilde E_6, \widetilde E_7$ or $\widetilde E_8$ as a subgraph.
In Example~\ref{4-cycle}, we have seen that the spectral radius of any Coxeter transformation associated to the $4$-cycle with alternating signs is greater than or equal to $3+2\sqrt{2}$.
With the same reasoning as in the proof of Proposition~\ref{affinelemma}, pairs of multicurves $\alpha$ and $\beta$ that intersect with the pattern of a graph that contains 
a $4$-cycle as a subgraph can be disregarded.

\begin{figure}[h]
\def\svgwidth{70pt}
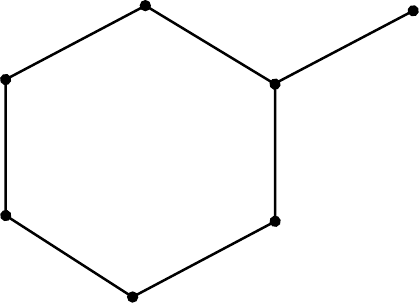
\caption{The enriched $6$-cycle.}
\label{6erMIT}
\end{figure}

The only intersection patterns of a pair of multicurves $\alpha$ and $\beta$ we still have to consider 
are the Dynkin diagrams $A_n$, $D_n$, $E_6$, $E_7$, $E_8$, the $2n$-cycle and the enriched $6$-cycle, depicted in Figure~\ref{6erMIT}.
By Lemma~\ref{Afilling}, Lemma~\ref{Dfilling} and Lemma~\ref{cyclefilling}, every $A_n$, $D_n$ and $2n$-cycle encoding the intersections of a pair of multicurves $\alpha$ and $\beta$ that fill $\Sigma_g$
contains $A_{2g}$ as a subgraph. In particular, the dilatations arising via Penner's construction using these 
multicurves are greater than or equal to the dilatation of the Coxeter mapping class $\phi_g$ associated to the Coxeter tree $(A_{2g}, \pm)$ with alternating signs. 
This proves Theorem~\ref{minimaldilatation} for $g\ge 5$, 
since the union of two multicurves $\alpha$ and $\beta$ intersecting with the pattern of $E_6$, $E_7$, $E_8$ or the enriched $6$-cycle can 
only fill a surface of genus~$g\le4$. The only thing left to deal with are these exceptional four graphs.
We note that the enriched $6$-cycle contains $E_7$ as a subgraph. 
Thus, the dilatation of any Coxeter mapping class associated to the enriched $6$-cycle with alternating signs is 
greater than or equal to the dilatation of the Coxeter mapping class associated to $(E_7,\pm)$. 
For these remaining four graphs, we simply calculate the dilatation of their associated alternating-sign Coxeter mapping classes
and compare them to the dilatation of the Coxeter mapping classes $\phi_g$ associated to $(A_{2g}, \pm)$. 
Table~1 sums up the situation.
\begin{table}[h]
\begin{tabular}{| c | c | c |}
\hline
graph & genus of surface filled & dilatation \\ \hline
$A_6$ & $3$ & $\approx 5.049$ \\ \hline
$A_8$ & $4$ & $\approx 5.345$  \\ \hline
$E_6$ & $3$ & $\approx 5.552$ \\ \hline
$E_7$ & $3$ & $\approx 5.704$ \\ \hline
$E_8$ & $4$ & $\approx 5.783$ \\ \hline
enriched $6$-cycle & $\le4$ & $> 5.7$\\ \hline
\end{tabular}
\smallskip
\caption{}
\end{table}
It is apparent that the Coxeter mapping classes associated to the Coxeter graphs $(A_{2g},\pm)$ with alternating signs minimise the dilatation 
also for closed surfaces of genus $g\le 4$. This completes the proof of Theorem~\ref{minimaldilatation}.

\section{Surfaces with punctures}
\label{punctures}
Let $\lambda_{g,p}$ be the minimal dilatation among mapping classes arising from Penner's construction for an orientable surface $\Sigma_{g,p}$ of genus $g$ with $p$ punctures. 
Up to now, we have determined $\lambda_{g,0}$. %Since every cell decomposition of a surface $\Sigma_{g,0}$ contains at least one $2$-cell, 
We remark that our proof works exactly the same for $\lambda_{g,1}$. 
%Actually, we have already blurred the distinction between $\Sigma_{g,0}$ and $\Sigma_{g,1}$: 
%the classical and alternating-sign Coxeter mapping classes associated to Dynkin trees $A_{2g}$ are usually defined on a once-punctured surface, see~\cite{AC2, Hi, HiLi}. 
%However, filling this puncture neither changes anything in Penner's construction for these mapping classes, nor does it change their homological actions.
This yields $\lambda_{g,0} = \lambda_{g,1}.$ 
If the number of punctures is small, say $p\le4$, it is conceivable that adjustments to our argument could be made, 
revealing alternating-sign Coxeter mapping classes associated to $A_n$, $D_n$ or the $2n$-cycle to minimise dilatation among mapping classes arising from Penner's construction. 
However, if the number of punctures increases, our examples cannot fill $\Sigma_{g,p}$ any longer and it seems that dilatations should become greater than $3+2\sqrt{2}$. 
In particular, our simplifications in the form of Propositions~\ref{doubleintersection} and~\ref{affinelemma} fail and many more cases would have to be considered: 
intersection patterns with affine subgraphs, loops, multiple edges and additional information on the framing of the corresponding curves. 

\begin{remark}
\label{finite}
\emph{The numbers $\lambda_{g,p}$ are bounded from above by a constant that does not depend on $g$ and $p$. 
%Indeed, for all numbers $g$ and $p$, we can construct examples with dilatation bounded from above by a constant $c$. 
Figure~\ref{7punctures} depicts two multicurves $\alpha$ and $\beta$ that intersect minimally and fill a sphere with eight punctures. 
Analogous examples can be constructed for any number of punctures.
\begin{figure}[h]
\def\svgwidth{120pt}
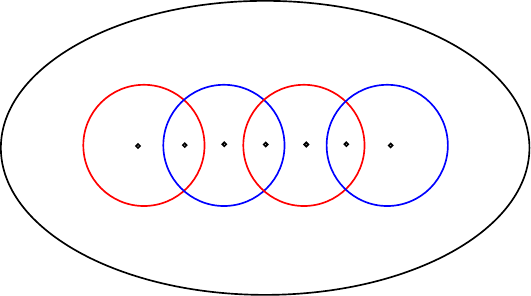
\caption{}
\label{7punctures}
\end{figure}
Furthermore, we can combine this example with our minimising examples for closed surfaces, as depicted in Figure~\ref{g3},
such that the components of the multicurves still intersect along a path.
Let $\phi$ be the product of Dehn twists along the components of the multicurves given by the bipartite order, with alternating signs. 
%Using that the components intersect along a path and at most two times, 
Since the multicurves intersect with the pattern of a path and each pair of components intersects at most twice,
it is a direct observation 
that there exists a constant $c$ such that every row sum of the matrix product $M_\phi$ of Penner's construction is bounded from above by $c$.
In particular, the Perron-Frobenius eigenvalue of $M_\phi$ and thus also the dilatation of $\phi$ is bounded from above by $c$.  
}\end{remark}

If $g$ is large compared to $p$, filling multicurves $\alpha$ and $\beta$ as in Remark~\ref{finite}, but with additional properties can be found. 
For example, multicurves $\alpha$ and $\beta$ with only single intersections among components. 
Furthermore, $\alpha$ and $\beta$ can be chosen to intersect with the pattern of a tree with vertices of degree at most three. 
This results in a smaller constant than $c$ from Remark~\ref{finite} bounding the row sums of the matrix product $M_\phi$ corresponding to the associated alternating-sign Coxeter mapping class $\phi$. 
\begin{remark}
\label{non-increasing}
\emph{
In general, for a fixed $p$, the function $\lambda_{g,p}$ is not increasing in $g$. Indeed, we have $$\lambda_{0,3}=3+2\sqrt{2}>\lambda_{1,3}.$$
The equality on the left follows from the following observation. If two simple closed curves on the sphere with three punctures intersect, then they intersect at least twice by the Jordan curve theorem. 
In particular, on the sphere with three punctures, pseudo-Anosov mapping classes arising from Penner's construction have $\lambda\ge3+2\sqrt{2}$ by Proposition~\ref{doubleintersection}.
On the other hand, it is possible to fill the sphere with three punctures by two simple closed curves intersecting exactly twice. The product of two Dehn twists along these curves realises $\lambda_{0,3}=3+2\sqrt{2}$.
The inequality on the right follows from %arguments similar to those used in the proof of Proposition~\ref{affinelemma} and 
the fact that the torus with three punctures can be filled by a pair of multicurves that intersect with the pattern of the Dynkin diagram $D_4$,
whose associated alternating-sign Coxeter transformation has spectral radius strictly smaller than $3+2\sqrt{2}$.
}\end{remark}
%It seems that $\lambda_{g,p}$ does not necessarily have to be non-decreasing in $g$. 
%While $\lambda_{g,p}$ clearly is non-decreasing in $p$, it seems possible that this does not necessarily hold for $g$.
Even though $\lambda_{g,p}$ is not always increasing in $g$, it might be so for $g$ large enough compared to $p$. 
In this case, one could ask about the limit of $\lambda_{g,p}$ for a fixed $p$, as $g\to +\infty$. These are two specific instances of the following, more broadly formulated question. 

\begin{que}
For a fixed $p>1$, what is the asymptotic behaviour of $\lambda_{g,p}$?
\end{que}

\end{document}

%% file: A3g.pdf_tex
%% Creator: Inkscape inkscape 0.48.2, www.inkscape.org
%% Accompanies image file 'A3g.pdf' (pdf, eps, ps)
\begingroup%
  \makeatletter%
  \providecommand\color[2][]{%
    \errmessage{(Inkscape) Color is used for the text in Inkscape, but the package 'color.sty' is not loaded}%
    \renewcommand\color[2][]{}%
  }%
  \providecommand\transparent[1]{%
    \errmessage{(Inkscape) Transparency is used (non-zero) for the text in Inkscape, but the package 'transparent.sty' is not loaded}%
    \renewcommand\transparent[1]{}%
  }%
  \providecommand\rotatebox[2]{#2}%
  \ifx\svgwidth\undefined%
    \setlength{\unitlength}{481.10022819bp}%
    \ifx\svgscale\undefined%
      \relax%
    \else%
      \setlength{\unitlength}{\unitlength * \real{\svgscale}}%
    \fi%
  \else%
    \setlength{\unitlength}{\svgwidth}%
  \fi%
  \global\let\svgwidth\undefined%
  \global\let\svgscale\undefined%
  \makeatother%
  \begin{picture}(1,0.23517222)%
    \put(0,0){\includegraphics[width=\unitlength]{A3g.pdf}}%
  \end{picture}%
\endgroup%

%% file: ADE.pdf_tex
%% Creator: Inkscape inkscape 0.48.2, www.inkscape.org
%% Accompanies image file 'ADE.pdf' (pdf, eps, ps)
\begingroup%
  \makeatletter%
  \providecommand\color[2][]{%
    \errmessage{(Inkscape) Color is used for the text in Inkscape, but the package 'color.sty' is not loaded}%
    \renewcommand\color[2][]{}%
  }%
  \providecommand\transparent[1]{%
    \errmessage{(Inkscape) Transparency is used (non-zero) for the text in Inkscape, but the package 'transparent.sty' is not loaded}%
    \renewcommand\transparent[1]{}%
  }%
  \providecommand\rotatebox[2]{#2}%
  \ifx\svgwidth\undefined%
    \setlength{\unitlength}{545.9879375bp}%
    \ifx\svgscale\undefined%
      \relax%
    \else%
      \setlength{\unitlength}{\unitlength * \real{\svgscale}}%
    \fi%
  \else%
    \setlength{\unitlength}{\svgwidth}%
  \fi%
  \global\let\svgwidth\undefined%
  \global\let\svgscale\undefined%
  \makeatother%
  \begin{picture}(1,0.94824364)%
    \put(0,0){\includegraphics[width=\unitlength]{ADE.pdf}}%
    \put(-0.00121626,0.93283253){\color[rgb]{0,0,0}\makebox(0,0)[lb]{\smash{$A_n$}}}%
    \put(0.00297012,0.7025815){\color[rgb]{0,0,0}\makebox(0,0)[lb]{\smash{$D_n$}}}%
    \put(0.0092497,0.46814409){\color[rgb]{0,0,0}\makebox(0,0)[lb]{\smash{$E_6$}}}%
    \put(0.01343608,0.23789302){\color[rgb]{0,0,0}\makebox(0,0)[lb]{\smash{$E_7$}}}%
    \put(0.01343608,0.00345561){\color[rgb]{0,0,0}\makebox(0,0)[lb]{\smash{$E_8$}}}%
  \end{picture}%
\endgroup%

%% file: 4cycle.pdf_tex
%% Creator: Inkscape inkscape 0.48.2, www.inkscape.org
%% Accompanies image file '4cycle.pdf' (pdf, eps, ps)
\begingroup%
  \makeatletter%
  \providecommand\color[2][]{%
    \errmessage{(Inkscape) Color is used for the text in Inkscape, but the package 'color.sty' is not loaded}%
    \renewcommand\color[2][]{}%
  }%
  \providecommand\transparent[1]{%
    \errmessage{(Inkscape) Transparency is used (non-zero) for the text in Inkscape, but the package 'transparent.sty' is not loaded}%
    \renewcommand\transparent[1]{}%
  }%
  \providecommand\rotatebox[2]{#2}%
  \ifx\svgwidth\undefined%
    \setlength{\unitlength}{177.38716431bp}%
    \ifx\svgscale\undefined%
      \relax%
    \else%
      \setlength{\unitlength}{\unitlength * \real{\svgscale}}%
    \fi%
  \else%
    \setlength{\unitlength}{\svgwidth}%
  \fi%
  \global\let\svgwidth\undefined%
  \global\let\svgscale\undefined%
  \makeatother%
  \begin{picture}(1,0.92419102)%
    \put(0,0){\includegraphics[width=\unitlength]{4cycle.pdf}}%
    \put(0.65341455,0.72305964){\color[rgb]{0,0,0}\makebox(0,0)[lb]{\smash{$v_1$}}}%
    \put(0.18373981,0.72241552){\color[rgb]{0,0,0}\makebox(0,0)[lb]{\smash{$v_2$}}}%
    \put(0.19018255,0.18444769){\color[rgb]{0,0,0}\makebox(0,0)[lb]{\smash{$v_3$}}}%
    \put(0.65341462,0.18058217){\color[rgb]{0,0,0}\makebox(0,0)[lb]{\smash{$v_4$}}}%
    \put(0.90854662,0.88992628){\color[rgb]{0,0,0}\makebox(0,0)[lb]{\smash{$+$}}}%
    \put(-0.00245504,0.88283936){\color[rgb]{0,0,0}\makebox(0,0)[lb]{\smash{$-$}}}%
    \put(-0.00374358,0.00662833){\color[rgb]{0,0,0}\makebox(0,0)[lb]{\smash{$+$}}}%
    \put(0.9008153,0.01951395){\color[rgb]{0,0,0}\makebox(0,0)[lb]{\smash{$-$}}}%
  \end{picture}%
\endgroup%

%% file: DEtilde.pdf_tex
%% Creator: Inkscape inkscape 0.48.2, www.inkscape.org
%% Accompanies image file 'DEtilde.pdf' (pdf, eps, ps)
\begingroup%
  \makeatletter%
  \providecommand\color[2][]{%
    \errmessage{(Inkscape) Color is used for the text in Inkscape, but the package 'color.sty' is not loaded}%
    \renewcommand\color[2][]{}%
  }%
  \providecommand\transparent[1]{%
    \errmessage{(Inkscape) Transparency is used (non-zero) for the text in Inkscape, but the package 'transparent.sty' is not loaded}%
    \renewcommand\transparent[1]{}%
  }%
  \providecommand\rotatebox[2]{#2}%
  \ifx\svgwidth\undefined%
    \setlength{\unitlength}{623.70222278bp}%
    \ifx\svgscale\undefined%
      \relax%
    \else%
      \setlength{\unitlength}{\unitlength * \real{\svgscale}}%
    \fi%
  \else%
    \setlength{\unitlength}{\svgwidth}%
  \fi%
  \global\let\svgwidth\undefined%
  \global\let\svgscale\undefined%
  \makeatother%
  \begin{picture}(1,0.8866667)%
    \put(0,0){\includegraphics[width=\unitlength]{DEtilde.pdf}}%
    \put(-0.00106471,0.74330507){\color[rgb]{0,0,0}\makebox(0,0)[lb]{\smash{$\widetilde{D_n}$}}}%
    \put(0.00443242,0.40981259){\color[rgb]{0,0,0}\makebox(0,0)[lb]{\smash{$\widetilde{E_6}$}}}%
    \put(0.00809717,0.20825117){\color[rgb]{0,0,0}\makebox(0,0)[lb]{\smash{$\widetilde{E_7}$}}}%
    \put(0.00809717,0.00302503){\color[rgb]{0,0,0}\makebox(0,0)[lb]{\smash{$\widetilde{E_8}$}}}%
  \end{picture}%
\endgroup%

%% file: 6cycle.pdf_tex
%% Creator: Inkscape inkscape 0.48.2, www.inkscape.org
%% Accompanies image file '6cycle.pdf' (pdf, eps, ps)
\begingroup%
  \makeatletter%
  \providecommand\color[2][]{%
    \errmessage{(Inkscape) Color is used for the text in Inkscape, but the package 'color.sty' is not loaded}%
    \renewcommand\color[2][]{}%
  }%
  \providecommand\transparent[1]{%
    \errmessage{(Inkscape) Transparency is used (non-zero) for the text in Inkscape, but the package 'transparent.sty' is not loaded}%
    \renewcommand\transparent[1]{}%
  }%
  \providecommand\rotatebox[2]{#2}%
  \ifx\svgwidth\undefined%
    \setlength{\unitlength}{201.03539336bp}%
    \ifx\svgscale\undefined%
      \relax%
    \else%
      \setlength{\unitlength}{\unitlength * \real{\svgscale}}%
    \fi%
  \else%
    \setlength{\unitlength}{\svgwidth}%
  \fi%
  \global\let\svgwidth\undefined%
  \global\let\svgscale\undefined%
  \makeatother%
  \begin{picture}(1,0.72258578)%
    \put(0,0){\includegraphics[width=\unitlength]{6cycle.pdf}}%
  \end{picture}%
\endgroup%

%% file: 7punctures.pdf_tex
%% Creator: Inkscape inkscape 0.48.2, www.inkscape.org
%% Accompanies image file '7punctures.pdf' (pdf, eps, ps)
\begingroup%
  \makeatletter%
  \providecommand\color[2][]{%
    \errmessage{(Inkscape) Color is used for the text in Inkscape, but the package 'color.sty' is not loaded}%
    \renewcommand\color[2][]{}%
  }%
  \providecommand\transparent[1]{%
    \errmessage{(Inkscape) Transparency is used (non-zero) for the text in Inkscape, but the package 'transparent.sty' is not loaded}%
    \renewcommand\transparent[1]{}%
  }%
  \providecommand\rotatebox[2]{#2}%
  \ifx\svgwidth\undefined%
    \setlength{\unitlength}{254.525bp}%
    \ifx\svgscale\undefined%
      \relax%
    \else%
      \setlength{\unitlength}{\unitlength * \real{\svgscale}}%
    \fi%
  \else%
    \setlength{\unitlength}{\svgwidth}%
  \fi%
  \global\let\svgwidth\undefined%
  \global\let\svgscale\undefined%
  \makeatother%
  \begin{picture}(1,0.55770553)%
    \put(0,0){\includegraphics[width=\unitlength]{7punctures.pdf}}%
  \end{picture}%
\endgroup%